\documentclass[reqno,11pt]{amsart}
\usepackage{amssymb,amsmath,amscd, amsthm, array}
\usepackage{cases} 
\usepackage{graphicx}
\usepackage{wrapfig}
\usepackage{color}
 \usepackage{bm}
 
\usepackage{newtxmath,newtxtext}
\usepackage{amsmath,mathtools}
\usepackage{amsfonts}
\usepackage{tikz-cd}
\usepackage{hyperref}
\hypersetup{colorlinks}
\hypersetup{
    colorlinks=true,       % false: boxed links; true: colored links
    linkcolor=red,          % color of internal links (change box color with linkbordercolor)
    citecolor=magenta,        % color of links to bibliography
    filecolor=violet,      % color of file links
    urlcolor=blue      % color of external links
}
\makeatletter
\@addtoreset{equation}{section}

\makeatother
  \makeatletter
  \def\@listi{%
  \leftmargin = 0pt \rightmargin = 0pt
     \labelwidth\leftmargin \advance\labelwidth-\labelsep
     \partopsep  = 0pt \itemsep       = 0pt
     \itemindent = 0pt \listparindent = 0pt}
  \let\@listI\@listi
  \@listi
  \def\@listii{%
     \leftmargin = 0pt \rightmargin = 0pt
     \labelwidth\leftmargin \advance\labelwidth-\labelsep
     \topsep     = 0pt \partopsep     = 0pt \itemsep   = 0pt
     \itemindent = 0pt \listparindent = 0pt}
  \let\@listiii\@listii
  \let\@listiv\@listii
  \let\@listv\@listii
  \let\@listvi\@listii
  \makeatother
  
\newtheorem{theorem}{Theorem}[section]
\newtheorem{corollary}{Corollary}[theorem]
\newtheorem{lemma}[theorem]{Lemma}
\newtheorem{conjectural theorem}{Conjectural theorem}[section]
\newtheorem*{definition}{Definition}

\newtheorem*{example}{Example}

\newcommand{\C}{{\mathbb{C}}}

\newcommand{\N}{{\mathbb{N}}}

\newcommand{\R}{{\mathbb{R}}}
\newcommand{\Z}{{\mathbb{Z}}}
\newcommand{\Fh}{{\mathcal{F}}}

\renewcommand{\det}{{\mathrm{det}}}

\newcommand{\tr}{{\mathrm{tr}}}
\newcommand{\Sp}{{\mathrm{Sp}}}
\newcommand{\Arg}{{\mathrm{Arg}}}

\begin{document}

    %전문삽입
    \title{On the leafwise cohomology and dynamical zeta functions for fiber bundles over the circle}

\author{Junhyeong Kim}

\address{Graduate School of mathematics \endgraf 
Kyushu University \endgraf
744 Motooka, Nishi-ku, \endgraf
Fukuoka 819-0395, Japan
}

\keywords{foliation, surface bundle, dynamical zeta function, trace formula
}

\thanks{ 
$^*$Research Fellow of the Japan Society for the Promotion of Science.
}

\email{j-kim@math.kyushu-u.ac.jp}

\subjclass[2010]{ 
Primary 53C12; Secondary 37C30.}

\begin{abstract}
In this paper, we give concrete descriptions of leafwise cohomology groups and show the regularized determinant expression of the dynamical zeta function for fiber bundles over $S^{1}$.
As applications, we show a functional equation and some formulas for special values of the dynamical zeta function.
\end{abstract}

    \maketitle
    %\tableofcontents

    %서문
    \section*{Introduction}

In a series of papers (c.f. [D1]-[D5]), Deninger considered smooth manifolds with 1-codimensional foliation and transverse flow as geometric analogues of arithmetic schemes.
He studied the dynamical zeta functions by means of the (infinite dimensional) leafwise cohomology groups.
Among other things, he presented a form of the conjectural dynamical Lefschetz trace formula, and showed dynamical analogues of the explicit formula in analytic number theory and the Lichtenbaum conjecture on special values of the Hasse-Weil zeta function of arithmetic schemes.
Although Deninger founded a general framework for his theory, concrete examples still remain to be explored.
So it would be interesting to investigate his theory for some concrete examples.

In this paper, we consider a standard and simple example, namely, a fiber bundle over $S^{1}$.
It is equipped with the 1-codimensional foliation $\Fh$ by the projection to the circle $S^{1}$ and the transverse flow suspended by the monodromy.
Moreover it may be regarded as a geometric analogue of an algebraic variety over a finite field. 
Our first result is the concrete description of the leafwise cohomology groups $H_{\Fh}^{\bullet}(M)$ and the infinitesimal generator acting on it in terms of the singular/de Rham cohomology of the fiber.
Using this, we show the regularized determinant expression for the dynamical zeta function. We note that such a expression follows from a "conjectural" dynamical Lefschetz trace formula for general foliated manifolds.
As applications, we show a functional equation and some formulas for special values of the dynamical zeta function of a fiber bundle over $S^{1}$.

    %본문
    \section{Preliminaries about foliated manifolds}\label{section1}
In this section, we recall fundamental settings; foliated dynamical systems on manifolds and leafwise cohomology groups.

\subsection{Foliation and transverse flow}

%foliation
Let $X$ be a smooth,connected, closed and oriented manifold of $n$-dimension with a $d$-codimensional foliation $\Fh=\Fh_{X}$. 
The formal definition of a foliation is as follows: Let $(U_{i})_{i\in{I}}$ be charts covering $X$ together with maps $(\varphi_{i}:U_{i}\rightarrow\R^{n})_{i\in{I}}$.
Assume that the transition maps $\varphi_{ij}:=\varphi_{j} \circ \varphi_{i}^{-1}$ which are defined over $U_{i}\cap U_{j}$ take the form $\varphi_{ij}(x,y)=(\varphi_{ij}^{1}(x),\varphi_{ij}^{2}(x,y))$ where $x$ denotes the first $d$ coordinates and $y$ denotes the last $n-d$ coordinates.
By piecing together the stripes, where $x$ is constant, from chart to chart, we obtain immersed sub-manifolds whose first $d$ local coordinates on each $U_{i}$ are constant.
They form a partition $\Fh$ of the manifold $X$, and we call it a $d$-codimensional foliation.
Each sub-manifold is called a $\mathit{leaf}$ of the foliation.

%transverse flow
Let $\phi:\R\times{X}\rightarrow{X}$ be a smooth $\R$-action such that maps leaves to leaves, i.e. for any two points $x$ and $y$ in a same leaf $\mathcal{L}$, there is a leaf $\mathcal{L}^{\prime}$ containing $\phi(t,x)$ and $\phi(t,y)$ for any $t\in\R$. 
We call the dynamical system the $\it{transverse}$ $\it{flow}$ which is compatible with the foliation $\Fh$. Denote $\phi(t,-):X\rightarrow{X}$ for $t\in\R$ by $\phi^{t}$. Note that each $\phi^{t}$ is a diffeomorphism of $X$.

\subsection{Leafwise cohomology and Infinitesimal generator}

%leafwise forms
Assuming that a manifold $X$ has a foliation $\Fh$, we have a sub-vector bundle $T\Fh$ of the tangent bundle $TX$ whose restriction to any leaf $\mathcal{L}$ is identified with the tangent bundle of the leaf, i.e. $T\Fh|_{\mathcal{L}}\cong T\mathcal{L}$. 
Differential $i$-forms along the leaves are defined as smooth sections to the sub-vector bundle $\wedge^{i} T^{*}\Fh$. 
Such a differential form is called a $\mathit{leafwise}$ $i$-$\mathit{form}$. 
Denote by $\mathcal{A}^{i}_{\Fh}(X)$ the space $\Gamma(X,\wedge^{i} T^{*}\Fh)$ of leafwise $i$-forms. 
If we consider a restriction map to a leaf $\mathcal{L}$, it induces the following map:
\begin{equation}
\label{eq:1.1}
\begin{matrix}
    \mathcal{A}^{i}_{\Fh}(X) & \rightarrow & \mathcal{A}^{i}(\mathcal{L}) \\
    \omega_{\Fh} &\mapsto &\omega_{\Fh}|_{\mathcal{L}}.
\end{matrix}
\end{equation}
It is easy to check that this map is surjective. 
The restriction of a leafwise form $\omega_{\Fh}$ becomes a differential form on the leaf. 

%leafwise cohomology
Let 
\begin{equation*}
d^{i}_{\Fh} : \mathcal{A}^{i}_{\Fh}(X) \rightarrow \mathcal{A}^{i+1}_{\Fh}(X).
\end{equation*}
be the exterior derivative which is restricted on $\mathcal{A}^{i}_{\Fh}(X)$. 
It is a differential operator which acts only along leaves direction. 
It satisfies the relation $d^{i+1}_{\Fh} \circ d^{i}_{\Fh} = 0$. 
Hence, the pairs $(\mathcal{A}^{i}_{\Fh}(X),d^{i}_{\Fh})$ form a complex. 
We call it the leafwise de Rham complex. We denote the kernel of $d_{\Fh}^{i}$ by $Z^{i}_{\Fh}(X)$ and the image of $d_{\Fh}^{i-1}$ by $B^{i}_{\Fh}(X)$. 
Note that leafwise $i$-forms in $Z^{i}_{\Fh}(X)$ (resp. $B^{i}_{\Fh}(X)$) are called leafwise closed $i$-forms (resp. leafwise exact $i$-forms).
Then, we have the following definition:
\begin{definition}
We define the $i$-th $\mathit{leafwise}$ $\mathit{cohomology}$ group of $\Fh$ by the $i$-th homology group of the leafwise de Rham complex.
\begin{equation*}
    H^{i}_{\Fh}(X) := Z^{i}_{\Fh}(X) / B^{i}_{\Fh}(X).
\end{equation*}
\end{definition}
Since $X$ is closed, the leafwise cohomology group is trivial for $i>n-d$.
For the transverse flow, each diffeomorphism $\phi^{t}$ can be induced on any $\mathcal{A}^{i}_{\Fh}(X)$ as the pullback. 
Then we have the following operator:
\begin{definition}
The infinitesimal generator on  $\mathcal{A}^{i}_{\Fh}(X)$ is defined by
\begin{equation*}
    \Theta := \underset{t\rightarrow0}{\mathrm{lim}} \frac{\phi^{t*}- \mathrm{id}}{t}.
\end{equation*}
\end{definition}
The operator $\Theta$ is shown to be independent of the choice of the time parameter $t$. 

    \section{Leafwise cohomology of fiber bundles over \texorpdfstring{$S^{1}$}{}}\label{section2}

In this section, we give a concrete description of the leafwise cohomology groups for fiber bundles over $S^{1}$.

\subsection{Fiber bundles over \texorpdfstring{$S^{1}$}{}}

%fiber bundles
Let $S$ be a orientable closed manifold of $d$ dimension and $\varphi$ be an orientation preserving diffeomorphism of $S$.
Assume that the number of periodic orbits of $\varphi$ is countably many.

\begin{example}
We call $\varphi$ an Anosov diffeomorphism if a tangent bundle splits into two sub-bundles which are invariant with respect to the differential of  $\varphi$, one of which is contracting and the other is expanding under $\varphi$ for some Riemannian metric.
It is known that if $\varphi$ is an Anosov diffeomorphism, it has countably many periodic points.

\end{example}

We define the mapping torus of $\varphi$ as follows:
\begin{equation} \label{eq:3.1}
    M:=S \times \R / (\varphi(x),t) \sim (x,t+\log{r})\quad\mathrm{for}\, r>1.
\end{equation}
The manifold $M$ is the fiber bundle over $S^{1}$ of $(d+1)$ dimension with the projection $\pi : M \rightarrow S^{1};[x,t]\mapsto{t\mbox{ mod }\Z}$.
We give additional structures on $M$; a foliation of 1-codimension and the transverse flow.

A foliation structure is given by the partition $\Fh=\{ \pi^{-1}(t) | t\in{S^{1}} \}$. Every fiber over the circle is leaves of the foliation.
Denote by $\mathcal{L}_{t}$ the leaf $\pi^{-1}(t)$ for $t\in{S^{1}}$.
The transverse flow $\phi$, which maps leaves to leaves, is defined by a smooth $\R$-action on $M$
\begin{equation*}
    \phi^{t}[x,s] = [x, s+ t]\quad\mbox{for}\,t\in\R.
\end{equation*}
The flow can be seen as the $\R$-action being suspended from the $\Z$-action on $S$ via the powers of $\varphi$.
Therefore we have a periodic obit $\mathfrak{o}$ of the $\Z$-action on $S$ for any closed orbit $\gamma$ of the $\R$-action on $M$.
The period of a closed orbit $\gamma$ is $\log{r}$ times the period $|\mathfrak{o}|$ of corresponding periodic orbit and the norm of $\gamma$ is defined by $\log N(\gamma):=|\mathfrak{o}|\log{r}$.

\subsection{The leafwise cohomology groups}

We describe leafwise differential forms more explicitly for $M$.
Choosing a fixed $t$ in $S^{1}$, we give a restriction map to the leaf  $\mathcal{L}_{t}=\pi^{-1}(t)$:
\begin{equation*}
\begin{matrix}
P_{i,t} : &\mathcal{A}^{i}_{\Fh}(M) &\rightarrow &\mathcal{A}^{i}(\mathcal{L}_{t})\\ &w_{\Fh} &\mapsto &w_{\Fh}|_{\mathcal{L}_{t}}.
\end{matrix}
\end{equation*}
Since $M$ is suspended by $\varphi$, leafwise differential forms have a natural boundary condition about the time parameter $t$; $\omega_{\Fh}|_{\mathcal{L}_{t+\log{r}}} = \varphi^{*}(\omega_{\Fh}|_{\mathcal{L}_{t}})$. 
It is easy to check that the $\R$-action $\phi^{\log{r} *}$ on $\mathcal{A}^{i}_{\Fh}(M)$ coincides with $\varphi^{*}$.

Leafwise differential forms can be described as paths in $\mathcal{A}^{i}(S)$ as follows: Denote by $\tilde{M}$ a $\Z$-covering of $M$ which is homeomorphic to $S\times\R$. 
It has the induced foliation $\tilde{\Fh}$ whose leaves are the slices $S\times \{ * \}$.
First, we have an injective map
\begin{equation}\label{eq:2.2}
    \begin{split}
        C^{\infty}(\R,\mathcal{A}^{i}(S)) &\hookrightarrow {\Gamma(S\times\R;\wedge^{i}T^{*}S)}=\mathcal{A}_{\tilde{\Fh}}^{i}(\tilde{M})\\
        [t\mapsto{s_{t}}] &\mapsto [(x,t)\mapsto{s_{t}(x)}].
    \end{split}
\end{equation}
Here we set a topology on $C^{\infty}(\R,\mathcal{A}^{i}(S))$ induced by $C^{\infty}$-topology on $\mathcal{A}_{\tilde{\Fh}}^{i}(\tilde{M})$.
Note that we also have an injective map
\begin{equation}\label{eq:2.3}
    \begin{split}
        \mathcal{A}_{\tilde{\Fh}}^{i}(\tilde{M})&\hookrightarrow{C}^{\infty}(\R,\mathcal{A}^{i}(S))\\
        \omega_{\tilde{\Fh}}&\mapsto{[t\mapsto{\omega_{\tilde{\Fh}}|_{S\times\left\{t\right\} }}]}.
    \end{split}
\end{equation}
By (\ref{eq:2.2}) and (\ref{eq:2.3}), we have the homeomorphism between $C^{\infty}(\R,\mathcal{A}^{i}(S))$ and $\mathcal{A}_{\tilde{\Fh}}^{i}(\tilde{M})$.

Leafwise differential forms on $M$ are injectively lifted into $\mathcal{A}_{\tilde{\Fh}}^{i}(\tilde{M})$ by the covering map.
Then we have the following map via the homeomorphism
\begin{equation*}
\begin{split}
    P : \mathcal{A}^{i}_{\Fh}(M) &\rightarrow {C}^{\infty}(\R,\mathcal{A}^{i}(S))\cong\mathcal{A}_{\tilde{\Fh}}^{i}(\tilde{M}) \\
    \omega_{\Fh}&\mapsto(t\mapsto\omega_{\Fh}|_{\mathcal{L}_{t}} ).
\end{split}
\end{equation*}
The image of $P$ is a subspace of ${C}^{\infty}(\R,\mathcal{A}^{i}(S))$ which satisfies the boundary condition.
We denote by ${C}^{\infty}_{\varphi}(\R,\mathcal{A}^{i}(S))$ the subspace.
We have the following homeomorphism:
\begin{equation*}
    P:\mathcal{A}^{i}_{\Fh}(M) \overset{\sim}{\longrightarrow} {C}^{\infty}_{\varphi}(\R,\mathcal{A}^{i}(S)).
\end{equation*}
We describe the leafwise cohomology groups as path spaces in the following theorem.
\begin{theorem}
\label{thm:2.1}
We have a homeomorphism:
\begin{equation} \label{eq_Psi}
    \begin{split}
            \Psi:{H}^{i}_{\Fh}(M) &\overset{\sim}{\longrightarrow} C^{\infty}_{\varphi}(\R,H^{i}(S))\\
            [\omega_{\Fh}]&\mapsto[t\mapsto[\omega_{\Fh}]|_{\mathcal{L}_{t}}].
    \end{split}
\end{equation}
\end{theorem}

\begin{proof}
Let
\begin{equation*}
\begin{matrix}
    {D}^{i}:&{C}^{\infty}(\R,\mathcal{A}^{i}(S)) &\rightarrow &{C}^{\infty}(\R,\mathcal{A}^{i+1}(S)) \\
     &(c:t \mapsto c(t)) &\mapsto & ({D}^{i}c : t \mapsto d^{i}_{s}c(t))
\end{matrix}
\end{equation*}
be an operator, where $d^{i}_{s}$ is the exterior derivative on $S$. 
Since it satisfies ${D}^{i} \circ {D}^{i+1}=0$, the pairs $({C}^{\infty}(\R,\mathcal{A}^{i}(S)), {D}^{i})$ form a cochain complex.

For a path $c$ in ${C}^{\infty}_{\varphi}(\R,\mathcal{A}^{i}(S))$, the operator ${D}^{i}$ is compatible with the boundary condition as follows:
\begin{equation*}
\begin{split}
    {D}^{i}c(t+\log{r}) &= d^{i}_{s}c(t+\log{r})\\
    &=d^{i}_{s}(\varphi^{*}(c(t))) \\
    &=\varphi^{*}(d^{i}_{s}c(t)) \\
    &=\varphi^{*}({D}^{i}c(t)).
\end{split}
\end{equation*}
Denote by ${D}^{i}_{r}$ the restricted operator.
It is easy to check that the restricted operator satisfies ${D}^{i}_{r} \circ {D}^{i+1}_{r}=0$. 
Hence the pairs $({C}^{\infty}_{\varphi}(\R,\mathcal{A}^{i}(S)),{D}^{i}_{r})$ form a cochain complex too.
Moreover, the homology groups of the cochain complex have the following isomorphisms
\begin{equation*}
    \begin{split}
        \mbox{Ker}(D^{i}_{r})/\mbox{Im}(D^{i-1}_{r})&\overset{\cong}{\rightarrow}{C}^{\infty}_{\varphi}(\R,{H}^{i}(S))\\
        {c \mbox{ mod }\mbox{Im}(D^{i-1}_{r})}&\mapsto[{t\mapsto{c(t)\mbox{ mod }\mbox{Im}(d^{i-1}_{s})}}].
    \end{split}
\end{equation*}

Finally, since the following diagram is commutative
\begin{center}
    \begin{tikzcd}
    \cdots\ar[r]&\mathcal{A}^{i}_{\Fh}(M) \ar[r, "d_{\Fh}^{i}"] \ar[d, "P"] & \mathcal{A}^{i+1}_{\Fh}(M) \ar[d, "P"]\ar[r]& \cdots\\
    \cdots\ar[r]&{C}^{\infty}_{\varphi}(\R,\mathcal{A}^{i}(S)) \ar[r, "D^{i}_{r}"] & {C}^{\infty}_{\varphi}(\R,\mathcal{A}^{i+1}(S))\ar[r] &\cdots,
\end{tikzcd}
\end{center}
the homeomorphism $p$ induces the quasi-isomorphism $\Psi$ between cochain complexes.
\end{proof}

Let $\frac{d}{dt}$ be a differential operator on $ C^{\infty}_{\varphi}(\R,H^{i}(S))$ given by
\begin{equation*}
(\frac{d}{dt}c)(t):=\lim_{h\rightarrow0}\frac{c(t+h)-c(t)}{h}\quad\mbox{for }t\in\R.
\end{equation*}
We have the following corollary for the operator:
\begin{theorem}
The infinitesimal generator $\Theta$ on  $\bar{H}^{i}_{\Fh}(M)$ corresponds to the differential operator $\frac{d}{dt}$ on $C^{\infty}_{\varphi}(\R,H^{i}(S))$:
\begin{equation*}
({H}^{i}_{\Fh}(M),\Theta) \cong (C^{\infty}_{\varphi}(\R,H^{i}(S)),\frac{d}{dt}).
\end{equation*}
\end{theorem}

\begin{proof}
For $[\omega_{\Fh}]\in{H}^{i}_{\Fh}(M)$, we have the following correspondence via the homeomorphism
    \begin{equation*}
        \Psi:\phi^{h*}[\omega_{\Fh}]\mapsto[t\mapsto[\omega_{\Fh}]|_{\mathcal{L}_{t+h}}].
    \end{equation*}
    Therefore we have $\Psi(\phi^{h*}[\omega_{\Fh}])(t)=\Psi([\omega_{\Fh}])(t+h)$.
    
    By using the property, we can show that the infinitesimal generator and the differential operator are commutative with $\Psi$
    \begin{equation*}
        \begin{split}
                \Psi(\Theta[\omega_{\Fh}])(t) &= \Psi\left(\lim_{h\rightarrow0} \frac{\phi^{h*}([\omega_{\Fh}]) - [\omega_{\Fh}]}{h}\right)(t)\\
                &= \lim_{h\rightarrow0}  \frac{ \Psi(\phi^{h*}([\omega_{\Fh}]))(t) - \Psi([\omega_{\Fh}])(t)}{h}\\
                &=\lim_{h\rightarrow0}  \frac{ \Psi([\omega_{\Fh}])(t+h) - \Psi([\omega_{\Fh}])(t)}{h}\\
                &=\left(\frac{d}{dt}\Psi([\omega_{\Fh}])\right)(t).
        \end{split}
    \end{equation*}
\end{proof}

    \section{Dynamical zeta functions of fiber bundles over \texorpdfstring{$S^{1}$}{}}\label{section3}

\subsection{Cohomological expression}
We define the dynamical zeta function of $M$ by
\begin{equation*}
    \zeta(M,\Fh_{M},\phi^{t};s):=\prod_{\gamma}(1-N(\gamma)^{-s})^{-1}
\end{equation*}
where $\gamma$ runs over periodic orbits of the $\R$-action $\phi^{t}$.
Denote simply by $\zeta(M;s)$ the dynamical zeta function of $(M,\Fh_{M},\phi^{t})$.
C. Deninger suggested that the dynamical zeta function can be described in terms of leafwise cohomology groups and infinitesimal generator and is an analogue of the Hasse-Weil zeta function.
We show that $\zeta(M;s)$ has the leafwise cohomological expression for $(M,\Fh_{M},\phi^{t})$ as C. Deninger conjectured.

\begin{theorem}
The dynamical zeta function $\zeta(M;s)$ is described in terms of leafwise cohomology groups and infinitesimal generator 
    \begin{equation*}
        \zeta(M;s)=\prod_{i}^{d}\det_{\infty}(s\cdot\mbox{id}-\Theta|H^{i}_{\Fh}(M,\C))^{(-1)^{i+1}}
    \end{equation*}
where $\det_{\infty}$ is the regularized determinant and $H^{i}_{\Fh}(M,\C)$ is the complexification of $H^{i}_{\Fh}(M)$.
\end{theorem}

We first see that the leafwise cohomological expression is defined for $(M,\Fh_{M},\phi^{t})$. 
Recall that regularized determinants $\det_{\infty}(\Theta|V)$ is defined if the following condition (1) and (2) hold:

(1) $V$ is a complex vector space of countable dimension.
An operator $\Theta$ has finitely many eigenvalues whose eigen-space is of finite dimension. Then $V$ is the direct sum of eigen-spaces.

(2) Under condition (1), let $\Sp(\Theta|V)$ be the set of eigenvalues of $\Theta$ with multiplicities.
We assume that the Dirichlet series
\begin{equation*}
    \sum_{\alpha\ne0\in\Sp(\Theta|V)}\frac{1}{\alpha^{s}}\mbox{ with }\alpha^{-s}=|\alpha|^{-s}e^{-is(\Arg{\alpha})},-\pi<\Arg{\alpha}\le\pi
\end{equation*}
converges absolutely for $\mbox{Re}{s}\gg0$ and has an analytic continuation to the half plane $\mbox{Re}{s}>-\epsilon$ for some $\epsilon>0$ which is holomorphic at $s=0$.

The following theorem implies that the regularized determinants $\det_{\infty}(s \cdot \mathrm{id} - \Theta | H^{i}_{\Fh}(M,\C))$ satisfy the condition (1) for $s\in\C$. 

\begin{theorem}
\label{thm:3.1}
$\Sp(\Theta|H^{i}_{\Fh}(M,\C))$ can be described in terms of $\Sp(\varphi^{*}|H^{i}(S,\C))$
\begin{equation}
\label{eq_eigenvalues}
    \Sp(\Theta|H^{i}_{\Fh}(M,\C))=\left\{ \frac{\log \alpha + 2{\pi}{i}{v}}{\log{r}} |\alpha\in\Sp(\varphi^{*}|H^{i}(S,\C)), v\in\Z \right\}.
\end{equation}

\end{theorem}

\begin{proof}
Let $[\omega_{\Fh}] \in H^{i}_{\Fh}(M,\C)$ be an eigenvector of the infinitesimal operator $\Theta$ with an eigenvalue $\alpha\in\C$:
\begin{equation*}
     \Theta [\omega_{\Fh}] =\alpha\cdot[\omega_{\Fh}].
\end{equation*}
By using Theorem \ref{thm:2.1}, we have the following differential equation:
\begin{equation*}
    \frac{d}{dt}\Psi([\omega_{\Fh}])(t)=\alpha\cdot\Psi([\omega_{\Fh}])(t)
\end{equation*}
Then the solution is given as a cohomological class on $S$.
\begin{equation*}
    \Psi([\omega_{\Fh}])(t)=e^{\alpha{t}}\Psi([\omega_{\Fh}])(0)
\end{equation*}
If we note that $\Psi([\omega_{\Fh}])(t+\log{r})=\varphi^{*} (\Psi([\omega_{\Fh}])(t))$, $\Psi([\omega_{\Fh}])(t)$ is an eigenvector of $\varphi^{*}$ with the eigenvalue $e^{\alpha\log{r}}=r^{\alpha}$ for a fixed $t$ in $S^{1}$:
\begin{equation*}
    \begin{split}
    \Psi([\omega_{\Fh}])(t+\log{r})&=e^{\alpha\log{r}} \Psi([\omega_{\Fh}])(t)\\
    &=\varphi^{*} (\Psi([\omega_{\Fh}])(t)).
    \end{split}
\end{equation*}
Thus, for an eigenvalue $\alpha \in \Sp(\Theta|H^{i}_{\Fh}(M,\C))$ , $r^{\alpha}=e^{\alpha \log{r}}$ should be the eigenvalue of $\varphi^{*}$ on the $H^{i}(S,\C)$.

Conversely, for an eigenvector $[\mu] \in H^{i}(S,\C)$  of $\varphi^{*}$ with an eigenvalue $\beta$, we suspend it to foliation cohomological classes $[\mu_{\Fh}]_{v}$ for all $v\in\Z$ which are defined by:
\begin{equation*}
    \Psi([\mu_{\Fh}]_{v})(t) := e^{(\frac{\log\beta+2\pi {i}{v}}{\log{r}})t}[\mu]\ \mathrm{for} \ t \in \R/\log{r}\,\Z,
\end{equation*}
where $i$ is $\sqrt{-1}$. They satisfy the boundary condition and are eigenvectors of $\Theta$ on $H^{i}_{\Fh}(M,\C)$ with the eigenvalues $\frac{\log \beta + 2 \pi{i}{v}}{\log{r}}$ for $\forall v \in \Z$. Therefore, $\{ \frac{\log \beta + 2 \pi{i}{v}}{\log{r}} |v \in \Z \}$ is the corresponding eigenvalues of the eigenvalue $\beta$.
\end{proof}
To show that the condition (2) holds, we interpret the Dirichlet series by using the Hurwitz zeta function 
\begin{equation}\label{eq:3.2.1}
\begin{split}
    \tr(s \cdot \mathrm{id} - \Theta | H^{i}_{\Fh}(M,\C))^{-z}&=\sum_{\alpha\in\Sp(\varphi^{*}_{i})}\sum_{v\in\Z}(s-\frac{\log{\alpha}+2\pi{i}{v}}{\log{r}})^{-z}\\
    &=\sum_{\alpha\in\Sp(\varphi^{*}_{i})}\sum_{v\in\Z} \left\{ \frac{2\pi{i}{v}}{\log{r}}\left(\frac{s\cdot\log{r}-\log \alpha}{2{\pi}{i}}+v\right) \right\}^{-z},
\end{split}
\end{equation}
where $\varphi^{*}_{i}$ denotes $\varphi^{*}$ on $H^{i}(S)$.

We recall that the Hurwitz zeta function $\zeta_{\mathfrak{hur}}(z,s)$ is  defined by:
\begin{equation*}
    \zeta_{\mathfrak{hur}}(z,s):=\sum_{n=0}^{\infty}(s+n)^{-z}\,\mathrm{for}\,\mathrm{Re}(z)>1,\,\mathrm{Re}(s)>0.
\end{equation*}
It is known that the series is absolutely convergent and can be extended to a meromorphic function defined for all $z\ne1$.
More generally we set
\begin{equation*}
    \zeta_{\eta,\mathfrak{hur}}(z,s)=\sum_{n=0}^{\infty}(\eta(s+n))^{-z},
\end{equation*}
where $\eta$ is a non-zero complex number.
If we denote  $\frac{2 \pi i}{\log{r}}$ and $\frac{s\cdot\log{r}-\log \alpha}{2{\pi}{i}}$ in (\ref{eq:3.2.1}) simply by $\eta$ and $s_{\alpha}$, respectively, then the series (\ref{eq:3.2.1}) is expressed by using the Hurwitz zeta function
\begin{equation*}
    \sum_{\alpha\in\Sp(\varphi^{*}_{i})}\zeta_{\eta,\mathfrak{hur}}(z,s_{\alpha})+\zeta_{-\eta,\mathfrak{hur}}(z,-s_{\alpha})-(\eta{s_{\alpha}})^{-z}.
\end{equation*}
Since the condition (2) holds from the property of the Hurwitz zeta function, the regularized determinant is defined for $(M,\Fh_{M},\phi^{t})$.

Next we show that the leafwise cohomological expression coincides with the dynamical zeta function of $(M,\Fh_{M},\phi^{t})$. We define a regularized product by 
\begin{equation*}
    \begin{split}
    \prod_{v\in\Z}\eta(s+v)&=\prod_{n\in\N}\eta(s+n)\cdot(-\eta)(-s+n)/(\eta{s})\\
    &:=\exp(-\partial_{z}\zeta_{\eta,\mathfrak{hur}}(0,s)-\partial_{z}\zeta_{-\eta,\mathfrak{hur}}(0,-s))\cdot(\eta{s})^{-1},
    \end{split}
\end{equation*}
where $\eta$ is a non-zero complex number. Then the lemma follows from properties of Hurwitz zeta function.
\begin{lemma}
\label{lem:3.2}
The regularized product such as the following form can be computed by using properties of the Hurwitz zeta functions:
\begin{equation*}
    \prod_{v\in\Z}\eta(s+v)=1-e^{-2\pi{i}s}
\end{equation*}
\end{lemma}
\begin{proof}
It is known for the Hurwitz zeta functions that:
\begin{equation*}
    \partial_{z}\zeta_{\mathfrak{hur}}(0,s)=\log\Gamma(s)-\frac{1}{2}\log2\pi.
\end{equation*}
Then we have
\begin{equation*}
    \exp(-\partial_{z}\zeta_{\eta,\mathfrak{hur}}(0,s))=\eta^{\frac{1}{2}-s}\left( \frac{\Gamma(s)}{\sqrt{2\pi}} \right)^{-1}.
\end{equation*}
By using the formula:
\begin{equation*}
    \frac{1}{s}\left( \frac{\Gamma(s)}{\sqrt{2\pi}} \right)^{-1}\left( \frac{\Gamma(-s)}{\sqrt{2\pi}} \right)^{-1}=e^{\pi{i}(\frac{1}{2}+s)}(1-e^{-2\pi{i}{s}}),
\end{equation*}
we get the following equality:
\begin{equation*}
    \begin{split}
        \prod_{v\in\Z}\eta(s+v)&=\eta^{\frac{1}{2}-s}\left( \frac{\Gamma(s)}{\sqrt{2\pi}} \right)^{-1}\cdot(-\eta)^{\frac{1}{2}+s}\left( \frac{\Gamma(-s)}{\sqrt{2\pi}} \right)^{-1}\cdot(\eta{s})^{-1}\\
        &=(\eta)^{-\frac{1}{2}+s}(-\eta)^{\frac{1}{2}+s}\cdot\frac{1}{s}\left( \frac{\Gamma(s)}{\sqrt{2\pi}} \right)^{-1}\left( \frac{\Gamma(-s)}{\sqrt{2\pi}} \right)^{-1}\\
        &=\begin{cases}
        1-e^{-2\pi{i}{s}}&\,\mathrm{If}\,\,0\leq\Arg{\eta}<\pi\\
        1-e^{2\pi{i}{s}}&\,\mathrm{If}\,\,-\pi\leq\Arg{\eta}<0
        \end{cases}
    \end{split}
\end{equation*}
In this case, since $0\leq\Arg{\eta}<\pi$, we get the statement of the lemma.
\end{proof}
We denote the leafwise cohomological expression by the regularized product as follows
\begin{equation*}
\begin{split}
    \prod_{i}^{d}\det_{\infty}(s\cdot\mbox{id}-\Theta|H^{i}_{\Fh}(M,\C))^{(-1)^{i+1}}&\overset{(\ref{eq_eigenvalues})}{=}\prod_{i}\prod_{\alpha\in\Sp(\varphi^{*}_{i})}\prod_{v\in\Z} (s-\frac{\log{\alpha}+2\pi{i}{v}}{\log{r}})^{(-1)^{i+1}}\\
    &=\prod_{i}\prod_{\alpha\in\Sp(\varphi^{*}_{i})}\prod_{v\in\Z}  \left\{ \frac{2\pi{i}{v}}{\log{r}}\left(\frac{s\cdot\log{r}-\log \alpha}{2{\pi}{i}}+v\right) \right\}^{(-1)^{i+1}}.
\end{split}
\end{equation*}
By using the Lemma \ref{lem:3.2}, we have the following product form
\begin{equation*}
    \prod_{i}^{d}\det_{\infty}(s\cdot\mbox{id}-\Theta|H^{i}_{\Fh}(M,\C))^{(-1)^{i+1}}=\prod_{i}\prod_{\alpha\in\Sp(\varphi^{*}_{i})}(1-\exp(\log\alpha-s\cdot\log{r}))^{(-1)^{i+1}}.
\end{equation*}
The corollary follows from the product form.
\begin{corollary}
\label{cor:3.2.1}
The leafwise cohomological expression is described with the cohomology of the fiber $S$:
\begin{equation}\label{eq:3.5}
    \prod_{i}^{d}\det_{\infty}(s\cdot\mbox{id}-\Theta|H^{i}_{\Fh}(M,\C))^{(-1)^{i+1}} =\prod_{i=0}^{d}\det(1-\varphi^{*}\cdot r^{-s}|H^{i}(S,\C))^{(-1)^{i+1}}.
\end{equation}
\end{corollary}

The relation between periodic orbits of $\Z$-action on $S$ and $\R$-action on $M$ gives the following lemma.

\begin{lemma}
the dynamical zeta function $\zeta(M;s)$ of $(M,\Fh_{M},\phi^{t})$ has the cohomological expression in terms of the cohomology groups of the fiber $S$
\begin{equation}\label{eq:3.6}
    \zeta(M;s) =\prod_{i=0}^{d}\det(1-\varphi^{*}\cdot r^{-s}|H^{i}(S,\C))^{(-1)^{i+1}}.
\end{equation}
\end{lemma}

\begin{proof}
We have the bijective between periodic orbits of $\Z$-action on $S$ and $\R$-action on $M$.
If we denote by $\mathfrak{o}$ a corresponding periodic orbit of $\Z$-action on $S$ for a periodic orbit $\gamma$ of $\R$-action on $M$, the norm $N(\gamma)$ is defined by $\log{N(\gamma)}=|\mathfrak{o}|\log{r}$ where $|\mathfrak{o}|$ is the period of $\mathfrak{o}$.
Therefore we have
\begin{equation*}
    \zeta(M;s)=\prod_{\mathfrak{o}}(1-e^{-s|\mathfrak{o}|\log{r}})^{-1},
\end{equation*}
where $\mathfrak{o}$ runs over periodic orbits of the $\Z$-action on $S$.
Then the lemma follows from the Lefschetz fixed-point theorem.
\end{proof}

\begin{proof}[Proof of theorem 3.1]
The assertion follows from the composite of (\ref{eq:3.5}) and (\ref{eq:3.6}).
\end{proof}

\subsection{Functional equations}

The zeta function $\zeta(M;s)$ has the functional equation as follows:
The closed manifold $S$ has the symmetric property which is given by a perfect pairing
\begin{equation*}
    \cup:H^{i}(S) \times H^{d-i}(S) \rightarrow H^{d}(S).
\end{equation*}
The pullback of the diffeomorphism $\varphi$ of $S$ is compatible with the pairing
\begin{equation*}
    \varphi^{*}(v)\cup\varphi^{*}(w)=\varphi^{*}(v\cup{w})\quad\mathrm{for}\,v\in{H}^{i}(S),w\in{H}^{d-i}(S).
\end{equation*}
Since $\varphi^{*}$ acts on $H^{d}(S)$ as the identity, the following holds (\cite{Ha},Appendix C, Lemma 4.3.):
\begin{equation}\label{eq:3.3}
    \det(\varphi^{*}|H^{d-i}(S))=\frac{1}{\det(\varphi^{*}|H^{i}(S))}
\end{equation}
and
\begin{equation}\label{eq:3.4}
    \det(1-\varphi^{*}t|H^{d-i}(S))=\frac{(-t)^{\beta(i)}}{\det(\varphi^{*}|H^{i}(S))}\det(1-\varphi^{*}t^{-1}|H^{i}(S)),
\end{equation}
where $\beta(i)$ is the $i$-th Betti number.
\begin{theorem}[Functional equation]
\label{functional_eq}
The zeta function is symmetrical about $Re(s)=0$
\begin{equation*}
    \zeta(M;s)=(-r^{s})^{\chi(S)} \cdot \zeta(M;-s),
\end{equation*}
where $\chi(S)$ is the Euler characteristic of $S$.
\end{theorem}
\begin{proof}
As the result of the corollary \ref{cor:3.2.1}, the zeta function can be described with the cohomology groups of $S$:
\begin{equation*}
    \begin{split}
            \zeta(M;s)&=\prod_{i=0}^{d}\det(1-\varphi^{*}\cdot r^{-s}|H^{d-i}(S,\C))^{(-1)^{i+1}}\\
            &\overset{(\ref{eq:3.4})}{=}\prod_{i=0}^{d}\left\{\frac{(-r^{-s})^{\beta(i)}}{\det(\varphi^{*}|H^{i}(S))}\det(1-\varphi^{*}r^{s}|H^{i}(S))\right \}^{(-1)^{i+1}}\\
            &\overset{(\ref{eq:3.3})}{=}(-r^{-s})^{-\chi(S)}\zeta(M;-s).
    \end{split}
\end{equation*}
\end{proof}

\subsection{Special values of \texorpdfstring{$\zeta(M;s)$}{}}

We compute special values of the dynamical zeta function $\zeta(M;s)$:
We define the special value of $\zeta(M;s)$ at $s=k$ by
\begin{equation*}
    \zeta(M;k)^{*}:=\lim_{s\rightarrow k}\zeta(M;s)\cdot (s-k)^{-\mathrm{ord}_{s=k}\zeta(M;s)},
\end{equation*}
where ${-\mathrm{ord}_{s=k}\zeta(M;s)}$ is the order of $\zeta(M;s)$ at $s=k$ for $k\in\C$.
Note that the order of $\zeta(M;s)$ at $s=k$ for $k\in\C$ is given by
\begin{equation*}
    \mathrm{ord}_{s=k}\zeta(M;s)=\sum_{i}(-1)^{i+1}\dim(H^{i}_{\Fh}(M)^{\Theta\sim k}),
\end{equation*}
where $H^{\sigma\sim{k}}$ is the eigenspace of $\sigma$ on $H$ whose eigenvalue is $k$.

We show the computation of special values of $\zeta(M;s)$ as follows.
\begin{theorem}
If we set the $m$-th Lefschetz number for $m\in\Z$ by 
\begin{equation*}
    \Lambda(\varphi^{m}):=\sum_{i}(-1)^{i}\mathrm{tr}(\varphi^{m*}|H^{i}(S)),
\end{equation*}        
the special value of $\zeta(M;s)$ at $s=k$ can be expressed  with the order of the function and the Lefschetz number.
\begin{equation*}
    \zeta(M;k)^{*}=(\log{r})^{\mathrm{ord}_{s=k}\zeta(M;s)}\exp(\sum_{m\ge1}\frac{r^{-km}\Lambda(\varphi^{m})+\mathrm{ord}_{s=k}\zeta(M;s)}{m}).
\end{equation*}
\end{theorem}
\begin{proof}
Corollary \ref{cor:3.2.1} allows $\zeta(M;s)$ to be decomposed with respect to eigenvalues of $\varphi^{*}$:
\begin{equation*}
    \zeta(M;s)=\prod_{i}(1-r^{k-s})^{(-1)^{i+1}\dim (H^{i}(S))^{\varphi^{*}\sim r^{k}}}\det(1-\varphi^{*}\cdot{r}^{-s}|H^{i}(S)^{\varphi^{*}\nsim r^{k}})^{(-1)^{i+1}}
\end{equation*}
We give the Taylor series about $(s-k)$
\begin{equation*}
    \begin{split}
    \zeta(M;s)=(\log{r}(s-k)+&O((s-k)^{2}))^{\mathrm{ord}_{s=k}\zeta(M;s)}\\
    &\cdot\exp(\sum_{m\ge1}\sum_{i}(-1)^{i}\mathrm{tr}(\varphi^{m*})\cdot \frac{r^{-sm}}{m}|H^{i}(S)^{\varphi^{*}\nsim r^{k}}).
    \end{split}
\end{equation*}
Note that $O(x^{2})$ is the big-O notation. Then we compute the special value at $s=k$
\begin{equation*}
    \zeta(M;k)^{*}=(\log{r})^{\mathrm{ord}_{s=k}\zeta(M;s)}\cdot\exp(\sum_{m\ge1}\sum_{i}(-1)^{i}\mathrm{tr}(\varphi^{m*})\cdot \frac{r^{-km}}{m}|H^{i}(S)^{\varphi^{*}\nsim r^{k}}).
\end{equation*}
By decomposing the Lefschetz number with respect to $\mathrm{ord}_{s=k}\zeta(M;s)$
\begin{equation*}
        \Lambda(\varphi^{m})=-r^{km}\cdot\mathrm{ord}_{s=k}\zeta(M;s)+\sum_{i}(-1)^{i}\mathrm{tr}(\varphi^{m*}|H^{i}(S)^{\varphi^{*}\nsim r^{k}}),
\end{equation*}
we get the statement of the theorem
\begin{equation*}
    \zeta(M;k)^{*}=(\log{r})^{\mathrm{ord}_{s=k}\zeta(M;s)}\exp(\sum_{m\ge1}\frac{r^{-km}\Lambda(\varphi^{m})+\mathrm{ord}_{s=k}\zeta(M;s)}{m}).
\end{equation*}
\end{proof}
    %\input{chapters/section4}

    %참고문헌
    
    %\addcontentsline{toc}{section}{References}

%끝
\end{document}